\documentclass[a4paper,12pt]{article}
\usepackage{latexsym}
\usepackage{amsmath}
\usepackage{amssymb}
\usepackage{enumerate}
\usepackage{bm}

\usepackage{theorem}
\newtheorem{theorem}{Theorem}[section]
\newtheorem{proposition}[theorem]{Proposition}
\newtheorem{lemma}[theorem]{Lemma}
\newtheorem{corollary}[theorem]{Corollary}

\theorembodyfont{\rmfamily}
\newtheorem{proof}{\textmd{\textit{Proof.}}}

\newtheorem{remark}[theorem]{Remark}

\newtheorem{definition}[theorem]{Definition}

\makeatletter

\@addtoreset{equation}{section}
\makeatother

\newcommand{\qedd}{\hfill \Box}
\newcommand{\ve}{\varepsilon}
\newcommand{\del}{\partial}
\newcommand{\lra}{\longrightarrow}

\newcommand{\N}{\ensuremath{\mathbb{N}}}
\newcommand{\Z}{\ensuremath{\mathbb{Z}}}
\newcommand{\R}{\ensuremath{\mathbb{R}}}

\newcommand{\cC}{\ensuremath{\mathcal{C}}}
\newcommand{\cL}{\ensuremath{\mathcal{L}}}
\newcommand{\bb}{\ensuremath{\mathbf{b}}}
\newcommand{\bK}{\ensuremath{\mathbf{K}}}
\newcommand{\bS}{\ensuremath{\mathbf{S}}}
\newcommand{\bL}{\ensuremath{\mathbf{L}}}

\def\vol{\mathop{\mathrm{vol}}\nolimits}
\def\div{\mathop{\mathrm{div}}\nolimits}

\def\loc{\mathop{\mathrm{loc}}\nolimits}
\def\Ric{\mathop{\mathrm{Ric}}\nolimits}

\def\Cut{\mathop{\mathrm{Cut}}\nolimits}
\def\Id{\mathop{\mathrm{Id}}\nolimits}

\def\CD{\mathop{\mathsf{CD}}\nolimits}
\def\MCP{\mathop{\mathsf{MCP}}\nolimits}
\def\Isom{\mathop{\mathsf{Isom}}\nolimits}

\newcommand{\Nabla}{\bm{\nabla}}
\newcommand{\Lap}{\bm{\Delta}}
\newcommand{\rev}[1]{\overleftarrow{#1}}
\newcommand{\wt}[1]{\widetilde{#1}}
\newcommand{\ol}[1]{\overline{#1}}

\title{Splitting theorems for Finsler manifolds\\ of nonnegative Ricci curvature}
\author{Shin-ichi Ohta\thanks{Department of Mathematics, Kyoto University,
Kyoto 606-8502, Japan ({\sf sohta@math.kyoto-u.ac.jp});
Supported in part by the Grant-in-Aid for Young Scientists (B) 23740048.}}
\date{\today}
\begin{document}

\maketitle

\begin{abstract}
We investigate the structure of a Finsler manifold
of nonnegative weighted Ricci curvature including a straight line,
and extend the classical Cheeger-Gromoll-Lichnerowicz splitting theorem.
Such a space admits a diffeomorphic, measure-preserving splitting in general.
As for a special class of Berwald spaces, we can perform the isometric splitting
in the sense that there is a one-parameter family of isometries
generated from the gradient vector field of the Busemann function.
A Betti number estimate is also given for Berwald spaces.

\end{abstract}

\section{Introduction}

The Ricci curvature is one of the most important quantities in geometry and analysis
on Riemannian manifolds.
The Ricci curvature (or tensor) plays prominent roles in various ways,
from the classical comparison theorems due to Rauch and Bishop
to Hamilton and Perelman's celebrated theory of the Ricci flow.
Recently, it turned out that the Ricci curvature is quite useful
also in the study of Finsler manifolds.
A Finsler manifold is a manifold endowed with a (Minkowski) norm on each tangent space.
Inspired by the theory of weighted Riemannian manifolds,
the \emph{weighted Ricci curvature} $\Ric_N$ was introduced in \cite{Oint}
for a Finsler manifold $(M,F)$ equipped with an arbitrary measure $m$ on $M$,
where $N \in [\dim M,\infty]$ is a parameter (see Definition~\ref{df:wRic}).
Bounding $\Ric_N$ from below by $K \in \R$ (i.e., $\Ric_N(v) \ge KF(v)^2$)
is equivalent to Lott, Sturm and Villani's \emph{curvature-dimension condition}
$\CD(K,N)$ (\cite{Oint}).
This equivalence has many applications via the general theory of the curvature-dimension condition,
such as the Bishop-Gromov volume comparison and the Lichnerowicz inequality
on the spectral gap (see \cite{Oint}).
Furthermore, the Laplacian comparison theorem for a natural nonlinear Laplacian (\cite{OShf})
as well as the Bochner-Weitzenb\"ock formula (\cite{OSbw}) hold.

The aim of this article is to generalize another fundamental theorem in comparison geometry
involving the Ricci curvature, Cheeger and Gromoll's \emph{splitting theorem} (\cite{CG1}),
that asserts that a Riemannian manifold $(M,g)$ of nonnegative Ricci curvature including
a straight line admits an isometric splitting $M=M' \times \R$.
This splitting theorem was extended to weighted Riemannian manifolds by Lichnerowicz
and others (\cite{Li}, \cite{FLZ}, \cite{WW}), and to (weighted) Alexandrov spaces by introducing
appropriate notions of the lower Ricci curvature bound (\cite{KS}, \cite{ZZ}).

In the Finsler case, as normed spaces (equipped with the Lebesgue measure)
have the nonnegative Ricci curvature, the isometric splitting can not be expected.
Nevertheless, our first main result (Corollary~\ref{cr:dsplit}) asserts that
a diffeomorphic and measure-preserving splitting holds for general Finsler manifolds.
The proof is essentially parallel to the Riemannian case, thanks to the Laplacian comparison
theorem in \cite{OShf}.
One can describe the splitting in more details for Berwald spaces.
Roughly speaking, a Berwald space is a Finsler manifold modeled by a single normed space
(see Proposition~\ref{pr:Ich}).
For example, Riemannian manifolds, normed spaces, and their products are Berwald spaces.
In the Berwald case, we can also split the metric in the sense that there exists an $l$-parameter family
of isometries $\varphi_p:M \lra M$, $p \in \R^l$, such that $\bigsqcup_{p \in \R^l}\varphi_p(\ol{M})=M$,
where the $(\dim M-l)$-dimensional submanifold $\ol{M} \subset M$ is characterized
by the property that any Busemann function is constant on $\ol{M}$
(see Theorem~\ref{th:Bsplit} for the precise statement).
A Betti number estimate along the lines of \cite{CG1}, \cite{CG2} is also given in the Berwald case
(Theorem~\ref{th:Betti}).

The study of Finsler manifolds has an extra importance from the view of the curvature-dimension condition.
As we mentioned above, Finsler manifolds give a nice class of model spaces
satisfying the curvature-dimension condition.
One can use this class as a test to see what properties of Riemannian manifolds of
Ricci curvature bounded below can be or can not be expected to hold
for general metric measure spaces satisfying the curvature-dimension condition
(in other words, to see either such a property depends only on the `Ricci curvature bound',
or it also requires that the space is `Riemannian').
For instance, the contraction property of the heat flow with respect to the $L^2$-Wasserstein distance
fails on (non-Riemannian) Finsler manifolds (\cite{OSnc}).
Recently, the contraction property was shown in \cite{AGS} for metric measure spaces
by assuming $\CD(K,\infty)$ and the linearity of the heat flow
(the linearity means that the space is `Riemannian' in a sense).
The Bochner-Weitzenb\"ock formula also holds in such a case (see also \cite{GKO}),
whereas it is unclear how to remove the linearity.

The article is organized as follows.
We review necessary notions in geometry and analysis on Finsler manifolds
in Section~\ref{sc:prel}.
Section~\ref{sc:Buse} is devoted to the study of Busemann functions.
We show splitting theorems in Sections~\ref{sc:dsplit}, \ref{sc:Bsplit}
in the general and Berwald cases, respectively.

\section{Geometry and analysis on Finsler manifolds}\label{sc:prel}

We review the basics of Finsler geometry (we refer to \cite{BCS} and \cite{Shlec} for further reading),
and introduce the weighted Ricci curvature and the nonlinear Laplacian
studied in \cite{Oint} and \cite{OShf} (see also \cite{GS}).
Throughout the article, let $M$ be a connected, $n$-dimensional $\cC^{\infty}$-manifold
without boundary such that $n \ge 2$.
We fix an arbitrary positive $\cC^{\infty}$-measure $m$ on $M$ as our base measure.

\subsection{Finsler manifolds}

Given a local coordinate $(x^i)_{i=1}^n$ on an open set $\Omega \subset M$,
we will always use the coordinate $(x^i,v^j)_{i,j=1}^n$ of $T\Omega$ such that
\[ v=\sum_{j=1}^n v^j \frac{\del}{\del x^j}\Big|_x \in T_xM
 \qquad \text{for}\ x \in \Omega. \]

\begin{definition}[Finsler structures]\label{df:Fstr}
A nonnegative function $F:TM \lra [0,\infty)$ is called
a \emph{$\cC^{\infty}$-Finsler structure} of $M$ if the following three conditions hold.
\begin{enumerate}[(1)]
\item(Regularity)
$F$ is $\cC^{\infty}$ on $TM \setminus 0$,
where $0$ stands for the zero section.

\item(\emph{Positive $1$-homogeneity})
It holds $F(cv)=cF(v)$ for all $v \in TM$ and $c>0$.

\item(\emph{Strong convexity})
The $n \times n$ matrix
\begin{equation}\label{eq:gij}
\big( g_{ij}(v) \big)_{i,j=1}^n :=
 \bigg( \frac{1}{2}\frac{\del^2 (F^2)}{\del v^i \del v^j}(v) \bigg)_{i,j=1}^n
\end{equation}
is positive-definite for all $v \in TM \setminus 0$.
\end{enumerate}
We call such a pair $(M,F)$ a \emph{$\cC^{\infty}$-Finsler manifold}.
\end{definition}

That is to say, $F|_{T_xM}$ is a smooth \emph{Minkowski norm} for every $x \in M$,
and $F$ varies smoothly also in the horizontal direction.
We will denote the \emph{unit tangent sphere bundle} by $UM:=TM \cap F^{-1}(1)$.
For $x,y \in M$, we define the \emph{distance} from $x$ to $y$ in a natural way by
\[ d(x,y):=\inf_{\eta} \int_0^1 F\big( \dot{\eta}(t) \big) \,dt, \]
where the infimum is taken over all $\cC^1$-curves $\eta:[0,1] \lra M$
such that $\eta(0)=x$ and $\eta(1)=y$.
We remark that our distance can be \emph{nonsymmetric} (namely $d(y,x) \neq d(x,y)$)
since $F$ is only positively homogeneous.
A $\cC^{\infty}$-curve $\eta$ on $M$ is called a \emph{geodesic}
if it is locally minimizing and has a constant speed
(i.e., $F(\dot{\eta})$ is constant).
See \eqref{eq:geod} below for the precise geodesic equation.
Given $v \in T_xM$, if there is a geodesic $\eta:[0,1] \lra M$
with $\dot{\eta}(0)=v$, then we define the \emph{exponential map}
by $\exp_x(v):=\eta(1)$.
We say that $(M,F)$ is \emph{forward complete} if the exponential
map is defined on whole $TM$.
Then by the Hopf-Rinow theorem any pair of points
is connected by a minimal geodesic (cf.\ \cite[Theorem~6.6.1]{BCS}).

For each $v \in T_xM \setminus 0$, the positive-definite matrix
$(g_{ij}(v))_{i,j=1}^n$ in \eqref{eq:gij} induces
the Riemannian structure $g_v$ of $T_xM$ as
\begin{equation}\label{eq:gv}
g_v\bigg( \sum_{i=1}^n a_i \frac{\del}{\del x^i}\Big|_x,
 \sum_{j=1}^n b_j \frac{\del}{\del x^j}\Big|_x \bigg)
 := \sum_{i,j=1}^n a_i b_j g_{ij}(v).
\end{equation}
This inner product is regarded as the best Riemannian approximation of $F|_{T_xM}$ in the direction $v$,
and plays a vital role in the Riemannian geometric approach to Finsler geometry.
A  geometric way of introducing $g_v$ is that the unit sphere of $g_v$ is tangent to that of $F|_{T_xM}$
at $v/F(v)$ up to the second order.
In particular, we have $g_v(v,v)=F(v)^2$.

For later convenience, we recall a useful fact on homogeneous functions.

\begin{theorem}{\rm (cf.\ \cite[Theorem~1.2.1]{BCS})}\label{th:Euler}
Suppose that a differentiable function $H:\R^n \setminus \{0\} \lra \R$ satisfies
$H(cv)=c^r H(v)$ for some $r \in \R$ and all $c>0$ and $v \in \R^n \setminus \{0\}$
$($that is, $H$ is \emph{positively $r$-homogeneous}$)$.
Then we have
\[ \sum_{i=1}^n \frac{\del H}{\del v^i}(v)v^i=rH(v) \qquad
 \text{for all}\ v \in \R^n \setminus \{0\}. \]
\end{theorem}

The \emph{Cartan tensor}
\[ A_{ijk}(v):=\frac{F(v)}{2} \frac{\del g_{ij}}{\del v^k}(v)
 \qquad \text{for}\ v \in TM \setminus 0 \]
is a quantity appearing only in the Finsler context.
Indeed, $A_{ijk}$ vanishes everywhere on $TM \setminus 0$
if and only if $F$ comes from a Riemannian metric.
As $g_{ij}$ is positively $0$-homogeneous on each $T_xM \setminus 0$,
Theorem~\ref{th:Euler} yields
\begin{equation}\label{eq:Av}
\sum_{i=1}^n A_{ijk}(v)v^i =\sum_{j=1}^n A_{ijk}(v)v^j
 =\sum_{k=1}^n A_{ijk}(v)v^k =0
\end{equation}
for all $v \in TM \setminus 0$ and $i,j,k=1,2,\ldots,n$.
Define the \emph{formal Christoffel symbol}
\[ \gamma^i_{jk}(v):=\frac{1}{2}\sum_{l=1}^n g^{il}(v) \bigg\{
 \frac{\del g_{jl}}{\del x^k}(v) +\frac{\del g_{lk}}{\del x^j}(v)
 -\frac{\del g_{jk}}{\del x^l}(v) \bigg\} \quad \text{for}\ v \in TM \setminus 0, \]
where $(g^{ij}(v))$ stands for the inverse matrix of $(g_{ij}(v))$.
We also introduce the \emph{geodesic spray coefficient}
and the \emph{nonlinear connection}
\[ G^i(v):=\sum_{j,k=1}^n \gamma^i_{jk}(v) v^j v^k, \quad
 N^i_j(v):=\frac{1}{2} \frac{\del G^i}{\del v^j}(v)
 \qquad \text{for}\ v \in TM \setminus 0, \]
and $G^i(0)=N^i_j(0):=0$ by convention.
Note that $G^i$ is positively $2$-homogeneous, so that
Theorem~\ref{th:Euler} implies $\sum_{j=1}^n N^i_j(v) v^j=G^i(v)$.
Following another representation of $N^i_j$ (with the help of Theorem~\ref{th:Euler}) will be used:
\begin{equation}\label{eq:Nij}
N^i_j(v) =\sum_{k=1}^n \gamma^i_{jk}(v)v^k
 -\frac{1}{F(v)}\sum_{k,l,m=1}^n A^i_{jk}(v)\gamma^k_{lm}(v) v^l v^m,
\end{equation}
where $A^i_{jk}:=\sum_{l=1}^n g^{il}A_{ljk}$.

By using $N^i_j$, the coefficients of the \emph{Chern connection} are given by
\begin{equation}\label{eq:Gamma}
\Gamma^i_{jk}:=\gamma^i_{jk}
 -\sum_{l,m=1}^n \frac{g^{il}}{F}(A_{jlm}N^m_k +A_{lkm}N^m_j -A_{jkm}N^m_l)
 \quad \text{on}\ TM \setminus 0.
\end{equation}
That is, the corresponding \emph{covariant derivative} of a vector field
$X=\sum_{i=1}^n X^i (\del/\del x^i)$ by $v \in T_xM$
with \emph{reference vector} $w \in T_xM \setminus 0$ is defined as
\begin{equation}\label{eq:covd}
D_v^w X(x):=\sum_{i,j=1}^n \bigg\{ v^j \frac{\del X^i}{\del x^j}(x)
 +\sum_{k=1}^n \Gamma^i_{jk}(w) v^j X^k(x) \bigg\} \frac{\del}{\del x^i}\Big|_x \in T_xM.
\end{equation}
Then the \emph{geodesic equation} is written as, with the help of \eqref{eq:Av},
\begin{equation}\label{eq:geod}
D_{\dot{\eta}}^{\dot{\eta}} \dot{\eta}(t)
 =\sum_{i=1}^n \big\{ \ddot{\eta}^i(t) +G^i \big( \dot{\eta}(t) \big) \big\}
 \frac{\del}{\del x^i} \Big|_{\eta(t)} =0.
\end{equation}
The following fact will be used in Section~\ref{sc:Bsplit}, we give a proof for completeness.

\begin{lemma}\label{lm:key}
If all integral curves of a non-vanishing $\cC^{\infty}$-vector field $V$
are geodesic, then we have
\[ D^V_V W=D^{g_V}_V W, \qquad D^V_W V=D^{g_V}_W V \]
for any differentiable vector field $W$, where $D^{g_V}$ stands for the covariant derivative
with respect to the Riemannian structure $g_V$ given as \eqref{eq:gv}.
\end{lemma}

\begin{proof}
To see the claim, it suffices to compare
$\sum_{j,k=1}^n \Gamma^i_{jk}(V) W^j V^k$ with the corresponding quantity for $g_V$.
On the one hand, we observe from \eqref{eq:Gamma} and \eqref{eq:Av} that
\[ \sum_{j,k=1}^n \Gamma^i_{jk}(V) W^j V^k
 = \sum_{j=1}^n \bigg\{ \sum_{k=1}^n \gamma^i_{jk}(V) V^k
 -\sum_{l,m=1}^n \frac{g^{il}(V)}{F(V)}A_{jlm}(V) G^m(V) \bigg\} W^j. \]
On the other hand, since
\[ \frac{\del[g_{jl}(V)]}{\del x^k}= \frac{\del g_{jl}}{\del x^k}(V)
 +\sum_{m=1}^n \frac{2}{F(V)} A_{jlm}(V) \frac{\del V^m}{\del x^k}, \]
the corresponding quantity for $g_V$ is
\[ \sum_{j,k=1}^n \bigg\{ \gamma^i_{jk}(V)
 +\sum_{l,m=1}^n \frac{g^{il}(V)}{F(V)}A_{jlm}(V) \frac{\del V^m}{\del x^k} \bigg\} W^j V^k. \]
Then the geodesic equation
\[ \sum_{k=1}^n \frac{\del V^m}{\del x^k} V^k+G^m(V)=0 \]
shows that they coincide.
$\qedd$
\end{proof}

\subsection{Weighted Ricci curvature}

The \emph{Ricci curvature} (as the trace of the \emph{flag curvature}) for a Finsler manifold
is defined by using the Chern connection.
Instead of giving the precise definition in coordinates,
we explain an elegant interpretation due to Shen (\cite[\S 6.2]{Shlec}, \cite[Lemma~2.4]{Shcdv}).

Given a unit vector $v \in U_xM$, we extend it to a $\cC^{\infty}$-vector field $V$
on a neighborhood of $x$ in such a way that every integral curve of $V$ is geodesic,
and consider the Riemannian structure $g_V$ induced from \eqref{eq:gv}.
Then the flag curvature $\bK(v,w)$ for $w \in T_xM$ linearly independent with $v$
coincides with the sectional curvature of the plane spanned by $v$ and $w$
with respect to $g_V$ (in particular, it is independent of the choice of $V$).
Similarly, the Ricci curvature $\Ric(v)$ of $v$ with respect to $F$ coincides with
the Ricci curvature of $v$ with respect to $g_V$.

Inspired by the above interpretation of the Ricci curvature
and the theory of weighted Riemannian manifolds,
the weighted Ricci curvature for the triple $(M,F,m)$ was introduced in \cite{Oint} as follows.

\begin{definition}[Weighted Ricci curvature]\label{df:wRic}
We first define the function $\Psi:UM \lra \R$ on the unit tangent sphere bundle
via the decomposition $m=e^{-\Psi(\dot{\eta})}\vol_{\dot{\eta}}$
along unit speed geodesics $\eta$,
where $\vol_{\dot{\eta}}$ denotes the Riemannian volume measure of $g_{\dot{\eta}}$.
Then, given a unit vector $v \in U_xM$ and the geodesic $\eta:(-\ve,\ve) \lra M$
such that $\dot{\eta}(0)=v$, we define the \emph{weighted Ricci curvature}
involving a parameter $N \in [n,\infty]$ by
\begin{enumerate}[(1)]
\item $\Ric_n(v):=\displaystyle
 \begin{cases} \Ric(v)+(\Psi \circ \dot{\eta})''(0) &\ \text{if}\ (\Psi \circ \dot{\eta})'(0)=0, \\
 -\infty &\ \text{if}\ (\Psi \circ \dot{\eta})'(0) \neq 0, \end{cases}$

\item $\Ric_N(v):=\Ric(v) +(\Psi \circ \dot{\eta})''(0) -\displaystyle\frac{(\Psi \circ \dot{\eta})'(0)^2}{N-n}\ $
for $N \in (n,\infty)$,

\item $\Ric_{\infty}(v):=\Ric(v) +(\Psi \circ \dot{\eta})''(0)$.
\end{enumerate}
We also set $\Ric_N(cv):=c^2 \Ric_N(v)$ for $c \ge 0$.
\end{definition}

\begin{remark}\label{rm:Psi}
Let us add comments to the above concise definition of $\Psi$.
Fix $v \in U_xM$ and extend it to a $\cC^{\infty}$-vector field $V$ on a neighborhood $\Omega$ of $x$
such that all integral curves of $V$ are geodesic.
We can decompose our base measure $m$ as $m=e^{-\psi}\vol_V$ on $\Omega$
by using a function $\psi$ on $\Omega$.
Then, since $V(\eta(t))=\dot{\eta}(t)$ along the geodesic $\eta$ with $\dot{\eta}(0)=v$,
$\psi \circ \eta$ depends only on $v$ (independent of the choice of $V$).
Thus $\Psi(v):=\psi(x)$ is well-defined.
\end{remark}

We will say that $\Ric_N \ge K$ holds for some $K \in \R$
if $\Ric_N(v) \ge KF(v)^2$ for all $v \in TM$.
We remark that $(\Psi \circ \dot{\eta})'(0)$ coincides with Shen's \emph{$\bS$-curvature}
$\bS(v)$ (see \cite[\S 7.3]{Shlec}).
Observe that $\Ric_N(v) \le \Ric_{N'}(v)$ holds for $N<N'$.
It was shown in \cite[Theorem~1.2]{Oint} that, for each $K \in \R$,
$\Ric_N \ge K$ is equivalent to Lott, Sturm and Villani's \emph{curvature-dimension condition} $\CD(K,N)$.
This equivalence extends the corresponding result on (weighted) Riemannian manifolds
(due to \cite{vRS}, \cite{Stcon}, \cite{StI}, \cite{StII}, \cite{LV1}, \cite{LV2}),
and has many analytic and geometric applications (see \cite{Oint}).

\begin{remark}\label{rm:wRic}
For a Riemannian manifold $(M,g,\vol_g)$ endowed with the Riemannian volume measure,
clearly we have $\Psi \equiv 0$ and hence $\Ric_N =\Ric$ for all $N \in [n,\infty]$.
In general, however, a Finsler manifold may not admit any measure $m$ satisfying $\bS \equiv 0$
(in other words, $\Ric_n>-\infty$), see \cite{ORand} for such an example.
This means that there is no nice reference measure in general, so that we began with an arbitrary measure.
\end{remark}

For later convenience, we introduce the following notations.

\begin{definition}[Reverse Finsler structure]\label{df:rev}
Define the \emph{reverse Finsler structure} $\rev{F}$ of $F$ by $\rev{F}(v):=F(-v)$.
We will put arrows $\leftarrow$ on those quantities associated with $\rev{F}$,
for example, $\rev{d}\!(x,y)=d(y,x)$, $\rev{\Nabla}u=-\Nabla(-u)$
and $\rev{\Ric}_N(v)=\Ric_N(-v)$.
\end{definition}

Note that $\Ric_N \ge K$ is equivalent to $\rev{\Ric}_N(v) \ge K\rev{F}(v)^2$,
so that the Ricci curvature bound is equivalent between $F$ and $\rev{F}$.
We say that $(M,F)$ is \emph{backward complete} if $(M,\rev{F})$ is forward complete.
The forward and backward completenesses are not mutually equivalent in general.

\subsection{Nonlinear Laplacian}

Let us denote by $\cL^*:T^*M \lra TM$ the \emph{Legendre transform}
associated with $F$ and its dual norm $F^*$ on $T^*M$.
Precisely, $\cL^*$ is sending $\alpha \in T_x^*M$ to the unique element $v \in T_xM$
such that $\alpha(v)=F^*(\alpha)^2$ and $F(v)=F^*(\alpha)$.
Note that $\cL^*|_{T^*_xM}$ becomes a linear operator only when $F|_{T_xM}$
is an inner product.
For a differentiable function $u:M \lra \R$, the \emph{gradient vector}
of $u$ at $x$ is defined as the Legendre transform of the derivative,
$\Nabla u(x):=\cL^*(Du(x)) \in T_xM$.
For a differentiable vector field $V$ on $M$ and $x \in M$ such that $V(x) \neq 0$,
we define $\Nabla V(x) \in T_x^*M \otimes T_xM$ by using the covariant derivative \eqref{eq:covd} as
\[ \Nabla V(v):=D^V_v V \in T_xM \qquad \text{for}\ v \in T_xM. \]
We also set $\Nabla^2 u(x):=\Nabla(\Nabla u)(x)$
for a twice differentiable function $u:M \lra \R$ and $x \in M$ such that $Du(x) \neq 0$.

Define the \emph{divergence} of a differentiable vector field $V=\sum_{i=1}^n V^i (\del/\del x^i)$
on $M$ with respect to the base measure $m$ by
\[ \div_m V:=\sum_{i=1}^n \bigg( \frac{\del V^i}{\del x^i} +V^i \frac{\del \Phi}{\del x^i} \bigg), \]
where we decomposed $m$ in coordinates as $dm=e^{\Phi} \,dx^1 dx^2 \cdots dx^n$.
The divergence can be rewritten (and extended to weakly differentiable vector fields) in the weak form as
\[ \int_M \phi \div_m V \,dm =-\int_M D\phi(V) \,dm
 \qquad \text{for all}\ \phi \in \cC_c^{\infty}(M). \]
Then we define the distributional \emph{Laplacian} of $u \in H^1_{\loc}(M)$ by
$\Lap u:=\div_m(\Nabla u)$ in the weak sense that
\[ \int_M \phi\Lap u \,dm:=-\int_M D\phi(\Nabla u) \,dm
 \qquad \text{for all}\ \phi \in \cC_c^{\infty}(M). \]
We remark that $H^1_{\loc}(M)$ is defined solely in terms of the differentiable structure of $M$.
As the Legendre transform is nonlinear, this Laplacian is a nonlinear operator
unless $F$ comes from a Riemannian metric.

The weighted Ricci curvature $\Ric_N$ works quite well with the nonlinear Laplacian.
Among others, we recall the Laplacian comparison theorem (\cite[Theorem~5.2]{OShf})
in the special case of nonnegative curvature, as well as
the Bochner-Weitzenb\"ock formula (\cite[Theorems~3.3, 3.6]{OSbw}).

\begin{theorem}[Laplacian comparison theorem]\label{th:Lcomp}
Let $(M,F)$ be forward or backward complete,
and assume that $\Ric_N \ge 0$ for some $N \in [n,\infty)$.
Then, for any $z \in M$, the function $u(x)=d(z,x)$ satisfies
\[ \Lap u(x) \le \frac{N-1}{d(z,x)} \]
point-wise on $M \setminus (\{z\} \cup \Cut_z)$, and in the distributional sense
on $M \setminus \{z\}$.
\end{theorem}

We denoted by $\Cut_z$ the \emph{cut locus} of $z$.
The cut locus is the set of \emph{cut points} $x=\exp_z(v)$ such that
$\eta(t)=\exp_z(tv)$ is minimal on $[0,1]$ but not minimal on $[0,1+\ve]$ for any $\ve>0$.

\begin{theorem}[Bochner-Weitzenb\"ock formula]\label{th:BW}
Given $u \in H^2_{\loc}(M) \cap \cC^1(M)$ with $\Lap u \in H^1_{\loc}(M)$,
we have
\[ \Delta^{\Nabla u} \bigg( \frac{F(\Nabla u)^2}{2} \bigg) -D(\Lap u)(\Nabla u)
 =\Ric_{\infty}(\Nabla u) +\| \Nabla^2 u \|_{HS(\Nabla u)}^2 \]
as well as
\[ \Delta^{\Nabla u} \bigg( \frac{F(\Nabla u)^2}{2} \bigg) -D(\Lap u)(\Nabla u)
 \ge \Ric_N(\Nabla u) +\frac{(\Lap u)^2}{N} \]
for $N \in [n,\infty]$ point-wise on $M \setminus \{x \in M\,|\, \Nabla u(x)=0\}$,
and in the weak sense on $M$.
Here $\Delta^{\Nabla u}:=\div_m \circ \nabla^{g_{\Nabla u}}$ is the linearized Laplacian
associated with $g_{\Nabla u}$, and $\|\cdot\|_{HS(\Nabla u)}$ stands for
the Hilbert-Schmidt norm with respect to $g_{\Nabla u}$.
\end{theorem}

To be precise, in the definition of $\Delta^{\Nabla u}$,
we replace $\Nabla u$ with a measurable, non-vanishing vector field $V$
such that $V(x)=\Nabla u(x)$ if $\Nabla u(x) \neq 0$.
We remark that $\Delta^{\Nabla u} u=\Lap u$ holds (\cite[Lemma~2.4]{OShf}).

\subsection{Berwald spaces}

We introduce an important and reasonable class of Finsler manifolds.

\begin{definition}[Berwald spaces]\label{df:Btype}
We say that a Finsler manifold $(M,F)$ is of \emph{Berwald type} (or a \emph{Berwald space})
if $\Gamma_{jk}^i$ is constant on $T_xM \setminus 0$ for every $x \in M$.
\end{definition}

Clearly Riemannian manifolds and (smooth) Minkowski normed spaces are of Berwald type.
Non-Riemannian, non-flat Berwald spaces can be easily constructed by taking
various kinds of products of Berwald spaces (cf.\ \emph{Descartes products} in \cite[\S 2]{Sz1}).
Berwald spaces enjoy several fine properties (see \cite[Chapter~10]{BCS}),
we recall two of them for later use (cf.\ \cite[Proposition~10.1.1, Theorem~10.6.2]{BCS}).

\begin{proposition}[Isometry of tangent spaces, \cite{Ic}]\label{pr:Ich}
Let $(M,F)$ be a Finsler manifold of Berwald type.
Then, for any $\cC^1$-curve $\eta:[0,1] \lra M$ with $\dot{\eta} \neq 0$,
the parallel transport along $\eta$ is a linear isometry between
$(T_{\eta(0)}M,F|_{T_{\eta(0)}M})$ and $(T_{\eta(1)}M,F|_{T_{\eta(1)}M})$.
\end{proposition}

We remark that, in Berwald spaces, the covariant derivative~\eqref{eq:covd}
is independent of the choice of a reference vector.
Thus the parallel transport is unambiguously defined.

\begin{theorem}[Szab\'o's rigidity, \cite{Sz1}]\label{th:Szabo}
Let $(M,F)$ be a connected Berwald surface.
Then the following dichotomy holds.
\begin{itemize}
\item If the flag curvature is identically $0$, then $F$ is locally Minkowskian everywhere.
\item If the flag curvature is not identically $0$, then $F$ is Riemannian everywhere.
\end{itemize}
\end{theorem}

Szab\'o also classified higher dimensional non-Riemannian Berwald metrics
by means of holonomy theory (\cite{Sz1}, \cite{Sz2}),
whereas such a classification is not really helpful to our purpose.
We also remark that the \emph{Busemann-Hausdorff measure} satisfies $\bS \equiv 0$
for Berwald spaces (cf.\ \cite[\S 7.3]{Shlec}), though this fact will not be used.

\section{Analysis of Busemann functions}\label{sc:Buse}

Let $(M,F)$ be forward complete in this section.
We begin the study of the splitting phenomenon with analyzing Busemann functions.

We call a geodesic $\eta:[0,\infty) \lra M$ a \emph{ray} if it is globally minimizing
and has the unit speed (for brevity), i.e., $d(\eta(s),\eta(t))=t-s$ for all $s<t$.
Given a ray $\eta$, the associated \emph{Busemann function} $\bb_{\eta}:M \lra \R$ is defined by
\[ \bb_{\eta}(x):=\lim_{t \to \infty}\big\{ t-d\big( x,\eta(t) \big) \big\}. \]
This limit indeed exists because the triangle inequality ensures, for any $s<t$,
\[ s-d\big( x,\eta(s) \big) \le s-\big\{ d\big( x,\eta(t) \big) -(t-s) \big\}
 =t-d\big( x,\eta(t) \big) \le d\big( \eta(0),x \big). \]
The triangle inequality also shows that $\bb_{\eta}$ is $1$-Lipschitz in the sense that
\begin{equation}\label{eq:1-Lip}
\bb_{\eta}(y)-\bb_{\eta}(x) \le d(x,y) \qquad \text{for all}\ x,y \in M,
\end{equation}
and hence $\bb_{\eta}$ is differentiable almost everywhere.

We say that another ray $\sigma:[0,\infty) \lra M$ is
\emph{asymptotic} to $\eta$, denoted by $\sigma \sim \eta$,
if there are sequences $\{t_i\}_{i \in \N} \subset [0,\infty)$ and $\{\sigma_i\}_{i \in \N}$
such that $\lim_{i \to \infty}t_i=\infty$, $\sigma_i:[0,d(\sigma(0),\eta(t_i))] \lra M$
is a minimal geodesic from $\sigma(0)$ to $\eta(t_i)$, and that
$\lim_{i \to \infty} \sigma_i(t)=\sigma(t)$ for all $t \ge 0$.
The next lemma is concerned with the fundamental properties of Busemann functions
(cf.\ \cite[Theorem~3.8.2]{SST}).
We give proofs for completeness as our distance is nonsymmetric.

\begin{lemma}\label{lm:asymp}
Let $\eta:[0,\infty) \lra M$ be a ray.
\begin{enumerate}[{\rm (i)}]
\item For any $x \in M$, there exists a ray $\sigma$ asymptotic to $\eta$
such that $\sigma(0)=x$.

\item For any ray $\sigma \sim \eta$ and $s \ge 0$, it holds
$\bb_{\eta}(\sigma(s))=\bb_{\eta}(\sigma(0))+s$.

\item If $\bb_{\eta}$ is differentiable at $x \in M$, then
$\sigma(s)=\exp_x(s\Nabla \bb_{\eta}(x))$ is a unique ray asymptotic to $\eta$
emanating from $x$.
\end{enumerate}
\end{lemma}

\begin{proof}
(i) As $(M,F)$ is forward complete, we can
choose a unit speed minimal geodesic $\sigma_i$ from $x$ to $\eta(i)$ for each $i \in \N$.
By extracting a subsequence denoted again by $\{\sigma_i\}_{i \in \N}$,
the initial tangent vector $\dot{\sigma}_i(0)$ converges to some unit vector $v \in U_xM$.
Then the ray $\sigma(s):=\exp_x(sv)$ is asymptotic to $\eta$ by construction.

(ii) Take $\{t_i\}_{i \in \N}$ and $\{\sigma_i\}_{i \in \N}$ as in the definition
of the asymptoticity.
It holds that
\[ \bb_{\eta}\big( \sigma(s) \big)
 = \lim_{i \to \infty} \big\{ t_i-d\big( \sigma(s),\eta(t_i) \big) \big\} \]
by the definition of $\bb_{\eta}$.
We can replace $\sigma(s)$ in the right hand side with $\sigma_i(s)$ since
\[ \big| d\big( \sigma(s),\eta(t_i) \big)-d\big( \sigma_i(s),\eta(t_i) \big) \big|
 \le \max \big\{ d\big( \sigma(s),\sigma_i(s) \big),d\big( \sigma_i(s),\sigma(s) \big) \big\}
 \to 0 \]
as $i \to \infty$.
Hence we have, by the choice of $\sigma_i$,
\begin{align*}
\bb_{\eta}\big( \sigma(s) \big)
&= \lim_{i \to \infty} \big\{ t_i-d\big( \sigma_i(s),\eta(t_i) \big) \big\}
 =\lim_{i \to \infty} \big\{ t_i-d\big( \sigma_i(0),\eta(t_i) \big) +s \big\} \\
&= \lim_{i \to \infty} \big\{ t_i-d\big( \sigma(0),\eta(t_i) \big) +s \big\}
 =\bb_{\eta}\big( \sigma(0) \big) +s.
\end{align*}

(iii) Recall from \eqref{eq:1-Lip} that $\bb_{\eta}$ is $1$-Lipschitz.
Then we deduce from (ii) that any ray $\sigma \sim \eta$
with $\sigma(0)=x$ must satisfy $\dot{\sigma}(0)=\Nabla \bb_{\eta}(x)$.
This completes the proof.
$\qedd$
\end{proof}

The following is a key analytic property of Busemann functions.
The proof is similar to \cite[Lemma~5.6]{KS} (see also \cite{EH}, \cite[Lemma~2.1]{FLZ})
thanks to the Laplacian comparison theorem (Theorem~\ref{th:Lcomp}).

\begin{proposition}\label{pr:subh}
Assume that $\Ric_N \ge 0$ for some $N \in [n,\infty]$,
and that $\Psi:UM \lra \R$ as in Definition~$\ref{df:wRic}$ is bounded above if $N=\infty$.
Then $\bb_{\eta}$ is subharmonic, namely $\Lap \bb_{\eta} \ge 0$ holds
in the distributional sense.
\end{proposition}

\begin{proof}
We first treat the case of $N<\infty$.
Fix an arbitrary bounded open set $\Omega \subset M$ and a nonnegative test function
$\phi \in H^1_0(\Omega)$.
Put $r_i(x):=-d(x,\eta(i))$ for $i \in \N$.
Note that $r_i$ is differentiable almost everywhere and $\Nabla r_i(x)$ coincides with
the initial vector of the unique unit speed minimal geodesic from $x$ to $\eta(i)$.
Thanks to Lemma~\ref{lm:asymp}(iii) (and the construction in (i)),
we find $\lim_{i \to \infty} \Nabla r_i(x)=\Nabla\bb_{\eta}(x)$
for $x$ at where $\bb_{\eta}$ is differentiable.
Thus we have, by the dominated convergence theorem,
\[ \lim_{i \to \infty} \int_{\Omega} D\phi(\Nabla r_i) \,dm
 =\int_{\Omega} D\phi(\Nabla \bb_{\eta}) \,dm. \]
In order to apply Theorem~\ref{th:Lcomp}, we observe (recall Definition~\ref{df:rev})
\[ \Nabla r_i =-\rev{\Nabla}(-r_i)
 =-\rev{\Nabla} \left[ \rev{d}\!\big( \eta(i),\cdot \big) \right]. \]
Hence Theorem~\ref{th:Lcomp} for $\rev{F}$ yields
\begin{equation}\label{eq:Lcomp}
\int_{\Omega} D\phi(\Nabla \bb_{\eta}) \,dm
 =\lim_{i \to \infty} \int_{\Omega} \phi \rev{\Lap}(-r_i) \,dm
 \le (N-1) \lim_{i \to \infty} \int_{\Omega} \frac{\phi}{-r_i} \,dm=0.
\end{equation}

As for $N=\infty$, we derive from the calculation with respect to the Riemannian structure
$g_{\Nabla r_i}$ that (cf.\ \cite[(2.1)]{FLZ}),
since $\rev{g}_{\rev{\Nabla}(-r_i)} =g_{-\rev{\Nabla}(-r_i)} =g_{\Nabla r_i}$
and all integral curves of $\Nabla r_i$ are geodesic (with respect to $F$),
\begin{align*}
\rev{\Lap}(-r_i)(x)
&\le -\frac{n-1}{r_i(x)} +\frac{2\Psi(\dot{\sigma}_i(0))}{r_i(x)}
 +\frac{2}{r_i(x)^2}\int_0^{-r_i(x)} \Psi(\dot{\sigma}_i) \,ds \\
&\le -\frac{1}{r_i(x)} \left\{ (n-1)-2\Psi\big( \dot{\sigma}_i(0) \big) +2\sup_{UM} \Psi \right\}
\end{align*}
for $x \in M \setminus (\{\eta(i)\} \cup \rev{\Cut}_{\eta(i)})$,
where $\sigma_i:[0,-r_i(x)] \lra M$ is the unique minimal geodesic
from $x$ to $\eta(i)$ (with respect to $F$).
Therefore \eqref{eq:Lcomp} is available with
\[ N_{\Omega}=n+2 \Big( \sup_{UM} \Psi -\inf_{U\Omega}\Psi \Big) \]
in place of $N$, and $\bb_{\eta}$ is subharmonic.
$\qedd$
\end{proof}

\section{A diffeomorphic splitting}\label{sc:dsplit}

From here on, let $(M,F)$ be both forward and backward complete,
and assume that $\Ric_N \ge 0$ for some $N \in [n,\infty]$
and $\Psi$ (defined in Definition~\ref{df:wRic}) is bounded above if $N=\infty$.
Suppose that $(M,F)$ admits a \emph{straight line},
that is, a geodesic $\eta:\R \lra M$ with $d(\eta(s),\eta(t))=t-s$ for all $s<t$.
Let us consider the two Busemann functions
\[ \bb_{\eta}(x):=\lim_{t \to \infty}\big\{ t-d\big( x,\eta(t) \big) \big\}, \qquad
 \bb_{\bar{\eta}}(x):=\lim_{t \to \infty}\big\{ t-d\big( \eta(-t),x \big) \big\}, \]
where $\bb_{\bar{\eta}}$ is precisely the Busemann function for the ray $\bar{\eta}(t):=\eta(-t)$,
$t \in [0,\infty)$, with respect to $\rev{F}$.

\begin{proposition}\label{pr:harm}
Let $\eta:\R \lra M$ be a straight line.
Then we have $\bb_{\eta} +\bb_{\bar{\eta}} \equiv 0$,
and $\bb_{\eta}$ and $\bb_{\bar{\eta}}$ are harmonic with respect to
$F$ and $\rev{F}$, namely $\Lap\bb_{\eta} =\rev{\Lap} \bb_{\bar{\eta}} \equiv 0$.
In particular, $\bb_{\eta}$ and $\bb_{\bar{\eta}}$ are $\cC^{\infty}$ and
$\Lap\bb_{\eta} =\rev{\Lap} \bb_{\bar{\eta}} \equiv 0$ in fact holds in the point-wise sense.
\end{proposition}

\begin{proof}
We immediately observe from the triangle inequality that $\bb_{\eta} +\bb_{\bar{\eta}} \le 0$.
Proposition~\ref{pr:subh} implies $\Lap\bb_{\eta} \ge 0$ as well as
$\rev{\Lap}\bb_{\bar{\eta}} \ge 0$ (note that $\rev{\Psi}(v)=\Psi(-v)$).
Therefore
\[ \Lap\bb_{\eta} \ge 0 \ge -\rev{\Lap} \bb_{\bar{\eta}} =\Lap(-\bb_{\bar{\eta}}), \]
while $\bb_{\eta} \circ \eta \equiv -\bb_{\bar{\eta}} \circ \eta$.
Hence the strong maximum principle (see \cite[Theorem~2-2]{Da}, \cite[Lemma~5.4]{GS})
yields $\bb_{\eta}=-\bb_{\bar{\eta}}$ and $\Lap\bb_{\eta} =\rev{\Lap} \bb_{\bar{\eta}} \equiv 0$.

As a harmonic function is a static solution to the heat equation,
$\bb_{\eta}$ is $\cC^{1,\alpha}$ by \cite[Theorem~4.9]{OShf} (see also \cite[Theorem~1.1]{GS}).
Furthermore, $\Nabla\bb_{\eta}$ does not vanish
since $F(\Nabla\bb_{\eta}) \equiv 1$ by Lemma~\ref{lm:asymp}(ii),
so that $\bb_{\eta}$ and $\bb_{\bar{\eta}}=-\bb_{\eta}$ are eventually $\cC^{\infty}$
(see \cite[Remark~4.10]{OShf}, \cite[Theorem~1.1]{GS}).
$\qedd$
\end{proof}

We say that a straight line $\sigma:\R \lra M$ is \emph{bi-asymptotic} to $\eta$
if $\sigma|_{[0,\infty)} \sim \eta|_{[0,\infty)}$ and if $\bar{\sigma}(s):=\sigma(-s)$
is asymptotic to $\bar{\eta}$ with respect to $\rev{F}$.
Combining Proposition~\ref{pr:harm} with Lemma~\ref{lm:asymp}(iii),
we observe the following.

\begin{lemma}\label{lm:biasym}
Let $\eta:\R \lra M$ be a straight line.
Then, for any $x \in M$, the geodesic $\sigma:\R \lra M$ with $\dot{\sigma}(0)=\Nabla\bb_{\eta}(x)$
is a unique straight line bi-asymptotic to $\eta$ such that $\sigma(0)=x$.
\end{lemma}

Lemma~\ref{lm:asymp}(ii) implies not only $\Nabla\bb_{\eta} \neq 0$ but also that
every integral curve of $\Nabla\bb_{\eta}$ is geodesic.
Therefore $\Ric_N(\Nabla\bb_{\eta})=\Ric_N^{g_{\Nabla\bb_{\eta}}}(\Nabla\bb_{\eta})$
and we can apply the Cheeger-Gromoll-Lichnerowicz splitting theorem (\cite{CG1}, \cite{Li})
to the weighted Riemannian manifold $(M,g_{\Nabla\bb_{\eta}},m)$.

\begin{proposition}[Isometric splitting of $(M,g_{\Nabla\bb_{\eta}},m)$]\label{pr:Rsplit}
If $(M,F)$ contains a straight line $\eta:\R \lra M$,
then $(M,g_{\Nabla\bb_{\eta}})$ splits isometrically as $M=M' \times \R$ with $M'=\bb_{\eta}^{-1}(0)$,
and $\Psi \circ \dot{\sigma}$ is constant on the line $\sigma(s)=(x,s) \in M' \times \R$ for each $x \in M'$.
\end{proposition}

\begin{proof}
We give a sketch of the proof for thoroughness,
see \cite[Theorem~1.1]{FLZ}, \cite[Theorem~6.1]{WW} for details.
Applying the Bochner-Weitzenb\"ock formula for $g_{\Nabla\bb_{\eta}}$ to $\bb_{\eta}$,
we deduce from $\Delta^{\Nabla\bb_{\eta}}\bb_{\eta}=\Lap\bb_{\eta} \equiv 0$ that
\begin{equation}\label{eq:RBW}
\Ric_{\infty}(\nabla\bb_{\eta}) +\| \nabla^2 \bb_{\eta} \|_{HS}^2
 =\Delta \bigg( \frac{|\nabla\bb_{\eta}|^2}{2} \bigg) -D(\Delta\bb_{\eta})(\nabla\bb_{\eta}) =0,
\end{equation}
where $\nabla$ and $\Delta$ are with respect to $g_{\Nabla\bb_{\eta}}$ and $m$.
Thus the hypothesis $\Ric_{\infty} \ge \Ric_N \ge 0$ shows $\nabla^2 \bb_{\eta} \equiv 0$,
namely $\nabla\bb_{\eta}$ is a parallel (and hence Killing) vector field.
Therefore the associated one-parameter family of transforms $\varphi_t:M \lra M$, $t \in \R$,
consists of isometries (with respect to $g_{\Nabla\bb_{\eta}}$)
and $M$ is isometric to $M' \times \R$ with $M':=\bb_{\eta}^{-1}(0)$.

In order to split the measure $m$, we observe from
\[ \| \nabla^2 \bb_{\eta} \|_{HS}^2
 \ge \frac{(\Delta\bb_{\eta} +D(\Psi \circ \nabla\bb_{\eta})(\nabla\bb_{\eta}))^2}{n} \]
that $D(\Psi \circ \nabla\bb_{\eta})(\nabla\bb_{\eta}) \equiv 0$
(see, for example, the calculation in the proof of \cite[Theorem~3.3]{OSbw}).
$\qedd$
\end{proof}

The boundedness of $\Psi$ is in fact necessary for splitting both $g_{\Nabla\bb_{\eta}}$ and $m$,
see \cite[Examples~2.1, 2.2]{WW} for simple interesting examples.
Recall that we used the boundedness only in Proposition~\ref{pr:subh}.

\begin{corollary}[Diffeomorphic splitting of $(M,m)$]\label{cr:dsplit} 
Assume that $(M,F)$ includes a straight line.
Then $(M,m)$ admits a diffeomorphic, measure-preserving splitting
$(M,m)=(M' \times \R,m' \times \bL^1)$, where $\bL^1$ is the one-dimensional
Lebesgue measure and $m':=m|_{M'}$.
\end{corollary}

To be precise, the map
\[ (M' \times \R,m' \times \bL^1) \ni (x,t)\ \longmapsto\ \varphi_t(x) \in (M,m) \]
is diffeomorphic and measure-preserving.
We abused the notation that $m'=m|_{M'}$ denotes the projection of $m$ to $M'$,
namely $m'(U):=m(U \times [0,1])$ for any Borel set $U \subset M'$.
It is unclear if this splitting procedure can be iterated,
because it seems difficult to determine the structures of $(M',F|_{TM'},m')$ as well as
$(M',g_{\Nabla\bb_{\eta}}|_{TM'},m')$ from the construction in Proposition~\ref{pr:Rsplit}
(for instance, we do not know if they have the nonnegative curvature).

\begin{remark}\label{rm:tsplit}
Corollary~\ref{cr:dsplit} would be compared with the structure theorem
for a (non-branching) metric measure space $(X,d,m)$ satisfying the \emph{measure contraction property}
$\MCP(K,N)$ (see \cite{Omcp}, \cite[Section~5]{StII}) with $N \in (1,\infty)$ and $K>0$,
as $\CD(K,N)$ implies $\MCP(K,N)$.
If $(X,d)$ attains the maximal diameter $\pi\sqrt{(N-1)/K}$,
then $(X,m)$ is represented as the spherical suspension of some topological measure space
(\cite[Section~5]{Omcp2}).
It is unknown whether a splitting theorem similar to Corollary~\ref{cr:dsplit} holds
for general metric measure spaces satisfying $\MCP(0,N)$ or $\CD(0,N)$.
\end{remark}

Though the following theorem is essentially included in Proposition~\ref{pr:Rsplit},
we state it separately for future convenience.

\begin{theorem}\label{th:para}
Suppose that $(M,F,m)$ is forward and backward complete and satisfies $\Ric_N \ge 0$
for some $N \in [n,\infty]$, and that $\Psi$ is bounded above if $N=\infty$.
If $(M,F)$ contains a straight line $\eta:\R \lra M$, then we have $\Nabla^2 \bb_{\eta} \equiv 0$.
\end{theorem}

\begin{proof}
Applying the Bochner-Weitzenb\"ock formula (Theorem~\ref{th:BW}) to the harmonic function
$\bb_{\eta}$, we obtain
\[ \Ric_{\infty}(\Nabla\bb_{\eta}) +\|\Nabla^2 \bb_{\eta}\|_{HS(\Nabla\bb_{\eta})}^2=0 \]
and hence $\Ric_{\infty} \ge \Ric_N \ge 0$ implies $\Nabla^2 \bb_{\eta} \equiv 0$.
(We remark that this formula in fact coincides with the Bochner-Weitzenb\"ock formula~\eqref{eq:RBW}
for $g_{\Nabla\bb_{\eta}}$ due to the fact that all integral curves of $\Nabla\bb_{\eta}$
are geodesic, see \cite[Remark~3.4]{OSbw}.
Thus we do not really need the formula in \cite{OSbw}, the formula in the Riemannian case is enough.)
$\qedd$
\end{proof}

One may be able to derive from $\Nabla^2 \bb_{\eta} \equiv 0$ some more information
on the structure of $(M,F)$, whereas we have succeeded only in the Berwald case
(discussed in the next section).

\section{Berwald case}\label{sc:Bsplit}

In this final section, we investigate a more detailed splitting phenomenon of Berwald spaces
(recall Definition~\ref{df:Btype}).
Throughout the section, let $(M,F,m)$ be a forward and backward complete Berwald space,
and assume that $\Ric_N \ge 0$ for some $N \in [n,\infty]$ and $\Psi$ is bounded above if $N=\infty$.
By the definition of Berwald spaces,
the covariant derivative~\eqref{eq:covd} does not depend on the choice of a reference vector,
so that we will omit reference vectors in this section.
In particular, the covariant derivative is linear in the sense that $D_v(W+X)=D_v W+D_v X$.

A subset $A \subset M$ is said to be \emph{totally convex}
if any minimal geodesic joining two points in $A$ is contained in $A$.
We say that $A \subset M$ is \emph{geodesically complete} if,
for any geodesic $\eta:(0,\ve) \lra M$ contained in $A$,
its extension $\eta:\R \lra M$ as a geodesic is still contained in $A$.

\begin{lemma}\label{lm:tconv}
Suppose that $(M,F)$ contains a straight line $\eta:\R \lra M$.
Then, given any geodesic $\xi:[0,l] \lra M$, we have $(\bb_{\eta} \circ \xi)'' \equiv 0$.
In particular, for each $t \in \R$, $\bb_{\eta}^{-1}(t)$ is totally convex and geodesically complete.
\end{lemma}

\begin{proof}
We observe
\[ (\bb_{\eta} \circ \xi)''=g_{\Nabla\bb_{\eta}}(D_{\dot{\xi}}(\Nabla\bb_{\eta}),\dot{\xi})
 +g_{\Nabla\bb_{\eta}}(\Nabla\bb_{\eta},D_{\dot{\xi}} \dot{\xi}) =0 \]
(see \cite[Exercises~10.1.1, 10.1.2]{BCS} for the first equality).
To be precise, the first term vanishes in general by Theorem~\ref{th:para},
while the second term vanishes only in Berwald spaces
(since the covariant derivatives have $\Nabla\bb_{\eta}$ as the reference vector).
We in particular find that $\bb_{\eta} \circ \xi$ is constant if $\bb_{\eta}(\xi(r))=\bb_{\eta}(\xi(s))$
for some $r \neq s$, so that $\bb_{\eta}^{-1}(t)$ is totally convex and geodesically complete.
$\qedd$
\end{proof}

Define $\varphi_t:M \lra M$, $t \in \R$, as the one-parameter family
of $\cC^{\infty}$-transforms generated from $\Nabla\bb_{\eta}$.
Precisely, $\del \varphi_t/\del t =\Nabla\bb_{\eta}(\varphi_t)$.

\begin{proposition}\label{pr:Bsplit}
Suppose that $(M,F)$ contains a straight line $\eta:\R \lra M$.
Then we have the following.

\begin{enumerate}[{\rm (i)}]
\item
For any $t \in \R$, $\varphi_t$ is a measure-preserving isometry
such that $\varphi_t(M_0)=M_t$, where we set $M_t:=\bb_{\eta}^{-1}(t)$.
Moreover, it holds $M=\bigsqcup_{t \in \R} \varphi_t(M_0)$.
\item
The $(n-1)$-dimensional submanifold $(M_0,F|_{TM_0},m_0)$ is again of Berwald type
and satisfies $\Ric_{N-1} \ge 0$, where $m_0:=m|_{M_0}$
$($as in Corollary~$\ref{cr:dsplit})$.
\item
Define the projection $\rho:M \lra M_0$ by $\rho(\varphi_t(x)):=x$ for $(x,t) \in M_0 \times \R$.
Then a curve $\xi:\R \lra M$ is geodesic if and only if the projections $\rho(\xi):\R \lra M_0$
and $\bb_{\eta}(\xi):\R \lra \R$ are geodesic.
\end{enumerate}
\end{proposition}

\begin{proof}
(i)
We have already seen in Corollary~\ref{cr:dsplit} that $\varphi_t$ is measure-preserving.
Given any $v \in T_xM$, the isometric splitting of $(M,g_{\Nabla\bb_{\eta}},m)$
(Proposition~\ref{pr:Rsplit}) shows that $V(t):=D\varphi_t(v)$ is a parallel vector field
with respect to $g_{\Nabla\bb_{\eta}}$ along the geodesic $\sigma(t)=\varphi_t(x)$.
Hence it follows from Lemma~\ref{lm:key} that
$D_{\dot{\sigma}}V=D_{\dot{\sigma}}^{\dot{\sigma}}V
=D_{\dot{\sigma}}^{g_{\Nabla\bb_{\eta}}}V \equiv 0$.
Thus we obtain
\[ \frac{\del}{\del t}\big[ F(V)^2 \big] =\frac{\del}{\del t}\big[ g_V(V,V) \big]
 =2g_V(D_{\dot{\sigma}}V,V) \equiv 0 \]
and $\varphi_t$ is isometric (we were again indebted to the fine property
$D_{\dot{\sigma}}^V V =D_{\dot{\sigma}}^{\dot{\sigma}} V$ of Berwald spaces).

(ii)
The total convexity in Lemma~\ref{lm:tconv} guarantees that $(M_0,F|_{TM_0})$ is of Berwald type
(via the characterization~(e) in \cite[Theorem~10.2.1]{BCS} for instance).
In order to see $\Ric_{N-1} \ge 0$,
recall from Corollary~\ref{cr:dsplit} that $m_0$ enjoys $m=m_0 \times \bL^1$.
Fix a unit vector $v \in U_xM_0$ and extend it to a vector field $V_0$ on a neighborhood
$U \subset M_0$ of $x$ such that all integral curves of $V_0$ are geodesic.
We further extend $V_0$ to $V$ on $U \times (-\ve,\ve) \subset M$ by $V(y,t)=(V_0(y),0) \in T_{(y,t)}M$.
Then all integral curves of $V$ are geodesic and we deduce
\[ \Ric_{N-1}^{M_0}(v) =\Ric_N^M\! \big( (v,0) \big) \ge 0 \]
from (iii) below (as in Corollary~\ref{cr:Bsplit}(ii)) and the definition of $\Ric_N$ (since $(N-1)-(n-1)=N-n$).

(iii)
We can split the geodesic equation~\eqref{eq:geod} for $M$ into those for $M_0$ and $\R$
by virtue of a special property of Berwald spaces.
Take an open set $U \subset M_0$ and a coordinate $(x^1,x^2,\ldots,x^n)$
of $U \times \R \subset M$ such that $\bb_{\eta}(x)=x^n$
and that $(\del/\del x^i)_{i=1}^n$ is orthonormal with respect to $g_{\Nabla\bb_{\eta}}$.
As $\Gamma^i_{jk}$ is constant on every tangent space $T_xM \setminus 0$,
let us denote it by $\Gamma^i_{jk}(x)$.
We shall calculate
\begin{align*}
\Gamma^i_{jk}(x) &=\Gamma^i_{jk}\big( \Nabla\bb_{\eta}(x) \big) \\
&= \gamma^i_{jk}\big( \Nabla\bb_{\eta}(x) \big)
 -\sum_{m=1}^n (A_{jim}N^m_k +A_{ikm}N^m_j -A_{jkm}N^m_i) \big( \Nabla\bb_{\eta}(x) \big).
\end{align*}
Note that
\[ \frac{\del g_{ij}}{\del x^n}(\Nabla\bb_{\eta})=0 \qquad \text{for}\ i,j=1,2,\ldots,n \]
by (i), and that
\[ \frac{\del g_{in}}{\del x^j}(\Nabla\bb_{\eta})=0 \qquad \text{for}\ i,j=1,2,\ldots,n \]
by $g_{\Nabla\bb_{\eta}}(\Nabla\bb_{\eta},TM_0)=0$ (if $i \neq n$)
and $F(\Nabla\bb_{\eta}) \equiv 1$ (if $i=n$).
Hence $\gamma^i_{jk}(\Nabla\bb_{\eta}(x))=0$ unless $i,j,k \neq n$.
In particular, it follows from \eqref{eq:Nij} that $N^i_j(\Nabla\bb_{\eta}(x))=0$
for all $i,j=1,2,\ldots,n$.
Therefore we have $\Gamma^i_{jk}(x)=0$ unless $i,j,k \neq n$, so that
\[ D_{\dot{\xi}}\dot{\xi} =\sum_{i=1}^{n-1} \bigg\{ \ddot{\xi}^i
 +\sum_{j,k=1}^{n-1}\Gamma^i_{jk}(\xi) \dot{\xi}^j \dot{\xi}^k \bigg\} \frac{\del}{\del x^i}\Big|_{\xi}
 +\ddot{\xi}^n \frac{\del}{\del x^n}\Big|_{\xi}. \]
Thus the geodesic equation is split and we complete the proof.
$\qedd$
\end{proof}

We remark that, different from the Riemannian case, one can not reconstruct
$(M,F)$ only from $(M_0,F|_{TM_0})$.
Indeed, given $x \in M_0$, all we know is $F|_{T_xM_0}$ and the fact that
$T_xM_0 \perp \Nabla\bb_{\eta}(x)$ with respect to $g_{\Nabla\bb_{\eta}}$.
They give us only a little information about $F|_{T_xM \setminus T_xM_0}$
(see the related discussion in Remark~\ref{rm:k<l} below).

We can iterate the procedure in Proposition~\ref{pr:Bsplit} and obtain the following.

\begin{corollary}\label{cr:Bsplit}
\begin{enumerate}[{\rm (i)}]
\item
There exists a $k$-parameter family of measure-preserving isometries $\varphi_p:M \lra M$, $p \in \R^k$,
and an $(n-k)$-dimensional totally convex, geodesically complete submanifold $M' \subset M$ such that
\begin{itemize}
\item $(M',F|_{TM'})$ does not contain a straight line$;$
\item $(M',F|_{TM'},m')$ is of Berwald type with $\Ric_{N-k} \ge 0$, where $m':=m|_{M'};$
\item $\bigsqcup_{p \in \R^k} \varphi_p(M')=M;$
\item $\varphi_{p+q}=\varphi_q \circ \varphi_p$ for any $p,q \in \R^k$.
\end{itemize}
In particular, $(M,m)$ admits a diffeomorphic, measure-preserving splitting
$(M,m)=(M' \times \R^k,m' \times \bL^k)$.

\item
For each $x \in M$, $\Sigma_x:=\{ \varphi_p(x) \}_{p \in \R^k}$ is a $k$-dimensional submanifold
of $M$ whose flag curvature $($with respect to the restriction of $F)$ vanishes everywhere.

\item
For any $x,y \in M'$, $(\Sigma_x,F|_{T\Sigma_x})$ is isometric to $(\Sigma_y,F|_{T\Sigma_y})$.
\end{enumerate}
\end{corollary}

\begin{proof}
(i)
There is nothing to prove if $(M,F)$ does not contain a straight line ($k=0$, $M'=M$).
If $(M,F)$ contains a straight line $\eta_1:\R \lra M$, then Proposition~\ref{pr:Bsplit} provides
a one-parameter family $\{\varphi^{(1)}_t\}_{t \in \R}$ of measure-preserving isometries and
an $(n-1)$-dimensional totally convex, geodesically complete
submanifold $M_1 \subset M$ of $\Ric_{N-1} \ge 0$.
Suppose that $M_1$ contains a straight line $\eta_2:\R \lra M_1$ again.
We remark that $\eta_2$ is a straight line also as a curve in $M$
thanks to Proposition~\ref{pr:Bsplit}(iii).
Similarly to $\varphi^{(1)}_t$, we obtain
$\varphi^{(2)}_s:M_1 \lra M_1$, $s \in \R$, and $M_2 \subset M_1$.
Now, define $\varphi_{(t,s)}:M=M_1 \times \R \lra M$ by
\[ \varphi_{(t,s)}(x,r) :=\varphi^{(1)}_t\big( \varphi^{(2)}_s(x),r \big)
 =\big( \varphi^{(2)}_s(x),r+t \big). \]
This map clearly preserves the measure $m$.
To see that $\varphi_{(t,s)}$ is isometric, it is sufficient to show that the map
$(x,r) \longmapsto (\varphi^{(2)}_s(x),r)$ is isometric.
Consider the vector field $V:=(\Nabla \bb_{\eta_2},0) \in TM_1 \times T\R =TM$.
Note that all integral curves of $V$ are straight lines and mutually bi-asymptotic
(compare two lines $\xi, \zeta \subset M_1$ and then $\xi$ and $\varphi_t \circ \xi$).
Therefore $V$ coincides with the gradient vector field of the Busemann function on $M$
for the line $\eta_2$ (see Lemma~\ref{lm:biasym}).
Hence the argument as in Proposition~\ref{pr:Bsplit}(i) ensures that $\varphi_{(t,s)}$ is isometric.
The relation $\varphi_{(t+t',s+s')}=\varphi_{(t,s)} \circ \varphi_{(t',s')}$ is straightforward
from the construction.
Iterating this procedure as far as possible provides the desired space $M'$ containing no straight line
as well as the family of isometries $\{ \varphi_p \}_{p \in \R^k}$.

(ii)
The flatness is a consequence of the rigidity theorem for Berwald surfaces (Theorem~\ref{th:Szabo}).
As $\Sigma_x$ is totally geodesic (in other words, \emph{locally} totally convex)
by Proposition~\ref{pr:Bsplit}(iii), it is of Berwald type and
we can apply Theorem~\ref{th:Szabo} to each two-dimensional subspace $\Pi$ of $\Sigma_x$
(precisely, $\Pi=\{\varphi_p(x)\}_{p \in P}$ for a two-dimensional affine subspace $P \subset \R^k$).
Proposition~\ref{pr:Bsplit}(iii) also verifies that $\Pi$ must be flat even if it is Riemannian, therefore $\Sigma_x$ is flat.

(iii)
Take a minimal geodesic $\xi:[0,1] \lra M'$ from $x$ to $y$ and $p \in \R^k$.
By Theorem~\ref{th:para}, the parallel transport along $\varphi_p \circ \xi$ sends
$\Nabla\bb_{\eta_i}(\varphi_p(x))$ to $\Nabla\bb_{\eta_i}(\varphi_p(y))$ for any $i=1,\ldots,k$.
Since parallel transports are linearly isometric in Berwald spaces (Proposition~\ref{pr:Ich})
and $T_{\varphi_p(x)}\Sigma_x=\mathrm{span}\{ \Nabla\bb_{\eta_i}(\varphi_p(x)) \,|\, i=1,\ldots,k \}$,
we conclude that $\Sigma_x$ is isometric to $\Sigma_y$.
$\qedd$
\end{proof}

We remark that flat Berwald spaces are necessarily locally Minkowskian (\cite[Proposition~10.5.1]{BCS}).
Hence, if $M'$ degenerates to a single point $\{x\}$ ($k=n$),
then $M=\Sigma_x$ is an $n$-dimensional Minkowski normed space.
In general, however, it is unclear from the infinitesimal discussion in Proposition~\ref{pr:Bsplit} and
Corollary~\ref{cr:Bsplit} if $\Sigma_x$ is \emph{globally} totally convex in $M$
(see also Remark~\ref{rm:k<l} below).
One may be able to split $M'$ again in a slightly different way as follows.

\begin{theorem}[Isometric splitting of Berwald spaces]\label{th:Bsplit}
Suppose that $(M,F,m)$ is of Berwald type and forward and backward complete,
and satisfies $\Ric_N \ge 0$ for some $N \in [n,\infty]$.
Assume also that $\Psi$ is bounded above if $N=\infty$.
Then the following hold.
\begin{enumerate}[{\rm (i)}]
\item
There exists an $l$-parameter family of measure-preserving isometries $\varphi_p:M \lra M$, $p \in \R^l$,
and an $(n-l)$-dimensional totally convex, geodesically complete submanifold $\ol{M} \subset M$
such that
\begin{itemize}
\item any Busemann function associated with a straight line in $M$ is constant on $\ol{M};$
\item $(\ol{M},F|_{T\ol{M}},\ol{m})$ is of Berwald type with $\Ric_{N-l} \ge 0$,
where $\ol{m}:=m|_{\ol{M}};$
\item $\bigsqcup_{p \in \R^l} \varphi_p(\ol{M})=M;$
\item $\varphi_{p+q}=\varphi_q \circ \varphi_p$ for any $p,q \in \R^l$.
\end{itemize}
In particular, $(M,m)$ admits a diffeomorphic, measure-preserving splitting
$(M,m)=(\ol{M} \times \R^l,\ol{m} \times \bL^l)$.

\item
For each $x \in M$, $\Sigma_x:=\{ \varphi_p(x) \}_{p \in \R^l}$ is an $l$-dimensional submanifold
of $M$ whose flag curvature vanishes everywhere.

\item
For any $x,y \in \ol{M}$, $(\Sigma_x,F|_{T\Sigma_x})$ is isometric to $(\Sigma_y,F|_{T\Sigma_y})$.
\end{enumerate}
\end{theorem}

\begin{proof}
(i)
We start from the splitting $(M,m)=(M' \times \R^k, m' \times \bL^k)$ in Corollary~\ref{cr:Bsplit},
and shall use the same notations.
Assume that there is a straight line $\eta:\R \lra M$ such that
the Busemann function $\bb_{\eta}$ is not constant on $M'$.
We put $M'':=\bb_{\eta}^{-1}(0) \cap M'$ and observe that it is a totally convex, geodesically complete,
$(n-k-1)$-dimensional submanifold of Berwald type.
We next split $M'$.
In the canonical coordinate $(x,r_1,\ldots,r_k)$ of $M' \times \R^k$,
$\dot{\eta}(0) \in T_{\eta(0)}M$ is written as $(v,a_1,\ldots,a_k)$ with nonzero $v$.
As $\Nabla\bb_{\eta}$ is a parallel vector field (and so are $\Nabla\bb_{\eta_1},\ldots,\Nabla\bb_{\eta_k}$),
the $T\R^k$-component of $\Nabla\bb_{\eta}$ is always $(a_1, \ldots, a_k)$.
Observe that the $TM'$-component of $\Nabla\bb_{\eta}$ is a parallel vector field as well
due to the linearity of the covariant derivative.
Hence, for any $p \in \R^k$ and $x \in M'$, $\Nabla\bb_{\eta}(x)$ and
$\Nabla\bb_{\eta}(\varphi_p(x))$ have the same $T_xM'$-components.
Therefore the one-parameter family of measure-preserving isometries $\psi_t:M \lra M$, $t \in \R$,
generated from $\Nabla\bb_{\eta}$ splits into
\[ \psi_t=(\psi_t^{(1)},\psi_t^{(2)}) \quad \text{with}\quad
 \psi_t^{(1)}:M' \lra M', \quad \psi_t^{(2)}:\R^k \lra \R^k, \]
and is written in the above coordinate as
\[ \psi_t(x,r_1,\ldots,r_k)=\big( \psi_t^{(1)}(x), r_1+ta_1,\ldots,r_k+ta_k \big). \]
Thus $\psi^{(1)}_t=[(\Id_{M'}, \psi^{(2)}_{-t}) \circ \psi_t]|_{M'}$
derives a diffeomorphic, measure-preserving splitting
$(M',m')=(M'' \times \R,m'|_{M''} \times \bL^1)$.

We remark that the geodesic equation of $M'$ splits into those of $M'' \times \eta(\R)$
and then of $M'' \times \R$ (via the projection of $\eta(\R)$ to $M'$ in the splitting $M' \times \R^k$).
Then we see that $(M'',F,m'|_{M''})$ satisfies $\Ric_{N-k-1} \ge 0$ similarly to Proposition~\ref{pr:Bsplit}(ii).
Iterating this construction as far as possible, we eventually obtain the desired submanifold $\ol{M}$.

(ii) and (iii) are shown in the same way as Corollary~\ref{cr:Bsplit}(ii), (iii).
$\qedd$
\end{proof}

Due to (i) above, starting from a point $x \in M$,
we can characterize $\ol{M} \ni x$ by $\ol{M}=\bigcap_{\eta} \bb_{\eta}^{-1}(0)$,
where $\eta$ runs over all straight lines parametrized as $\bb_{\eta}(x)=0$.

\begin{remark}\label{rm:k<l}
Obviously $k \le l$ holds for $k$ in Corollary~\ref{cr:Bsplit} and $l$ in Theorem~\ref{th:Bsplit}.
The author does not know any example satisfying $k<l$.
We can verify $k=l$ if, for $\eta$ in the proof of Theorem~\ref{th:Bsplit}(i),
the projection $\rho(\eta):\R \lra M'$ is a straight line
(note the difference between $\eta_2 \subset M_1$
in Corollary~\ref{cr:Bsplit}(i) and $\eta \not\subset M'$ in Theorem~\ref{th:Bsplit}(i)).
The straightness is clear in the Riemannian case by the isometry of the splitting,
however, unclear in the Finsler setting from our infinitesimal argument
(we only know that $\rho(\eta)$ is geodesic by Proposition~\ref{pr:Bsplit}(iii)).
The difficulty comes from the fact that a (Minkowski) normed space $(\R^n,|\cdot|)$
may contain a triple $v,w_1,w_2 \in \R^n \setminus \{0\}$ such that
\[ g_v(v,w_1)=g_v(v,w_2)=0, \quad |w_1|<|w_2|, \quad |w_1+v|>|w_2+v|. \]
\end{remark}

Having Theorem~\ref{th:Bsplit} at hand,
we can obtain some topological information of compact Berwald spaces
along the lines of Cheeger and Gromoll's classical theory (\cite{CG1}, \cite{CG2}).

\begin{theorem}[A Betti number estimate]\label{th:Betti}
Let $(M,F,m)$ be a compact Finsler manifold of Berwald type satisfying
$\Ric_N \ge 0$ for some $N \in [n,\infty]$.
Denote by $(\wt{M},\wt{F})$ the universal Finsler covering of $(M,F)$
and by $\wt{M}=\ol{M} \times \R^l$ its splitting obtained in Theorem~$\ref{th:Bsplit}$.
Then we have the following.
\begin{enumerate}[{\rm (i)}]
\item
$\ol{M}$ is compact.
\item
There exists a finite normal subgroup $\Xi$ of the fundamental group $\pi_1(M)$
such that $\pi_1(M)/\Xi$ contains $\Z^l$ as a normal subgroup of finite index.
\item
The first Betti number of $M$ is at most $l$.
\end{enumerate}
\end{theorem}

\begin{proof}
Let us first give a remark on the action of the isometry group $\Isom(\wt{M},\wt{F})$.
For each $\Phi \in \Isom(\wt{M},\wt{F})$ and any straight line $\eta:\R \lra \wt{M}$,
it clearly holds $\bb_{\Phi \circ \eta} \circ \Phi=\bb_{\eta}$.
Hence $\Phi(\ol{M})$ inherits the property that all Busemann functions are constant.

(i)
Let $\Gamma \subset \Isom(\wt{M},\wt{F})$ be the deck transformation group of
the universal covering $\pi:\wt{M} \lra M$.
Since $M$ is compact, we can take a compact fundamental domain $K \subset \wt{M}$
such that $\Gamma(K)=\wt{M}$.
Now, if $\ol{M}$ is not compact, then there is a ray $\eta:[0,\infty) \lra M$ contained in $\ol{M}$.
For each $i \in \N$, we can choose $\Phi^{(i)} \in \Gamma$ satisfying $\Phi^{(i)}(\eta(i)) \in K$.
Then $\eta_i(t):=\Phi^{(i)}(\eta(i+t))$, $t \in [-i,\infty)$,
is a globally minimizing geodesic with $\eta_i(0) \in K$, and the compactness of $K$ ensures that
$\{\eta_i\}_{i \in \N}$ has a subsequence convergent to some geodesic $\eta_{\infty}:\R \lra M$
($\dot{\eta}_i(0) \to \dot{\eta}_{\infty}(0)$ to be precise).
This geodesic $\eta_{\infty}$ is a straight line by construction, and the Busemann function $\bb_{\eta_{\infty}}$
is not constant on $\Phi^{(i)}(\ol{M})$ for large $i$ (since $\dot{\eta}_i \in T[\Phi^{(i)}(\ol{M})]$).
This is a contradiction, so that $\ol{M}$ is compact.

(ii)
Fix $\Phi \in \Isom(\wt{M},\wt{F})$.
On the one hand, for each $p \in \R^l$, $\Phi(\varphi_p(\ol{M}))$
coincides with $\varphi_q(\ol{M})$ for some $q \in \R^l$
from the construction of $\ol{M}$ (since $\Phi$ preserves Busemann functions).
On the other hand, for each $x \in \ol{M}$, we find that
$\Phi(\Sigma_x)=\Sigma_y$ for some $y \in \ol{M}$ by the same reasoning.
Then the projection
\[ \Phi_1:=\varphi_{p-q} \circ \Phi|_{\varphi_p(\ol{M})}
 \in \Isom\big( \varphi_p(\ol{M}),\wt{F} \big) =\Isom(\ol{M},\wt{F}) \]
(we identified $\Isom(\varphi_p(\ol{M}),\wt{F})$ with $\Isom(\ol{M},\wt{F})$
via $\varphi_p$) is independent of the choice of $p \in \R^l$
because $\Phi$ preserves the $\R^l$-directions (that is, $\Phi(\Sigma_x)=\Sigma_y$ and $\Phi_1(x)=y$).
We similarly see that $\Phi_2:=\Phi|_{\Sigma_x} \in \Isom(\Sigma_x,\wt{F})$
is independent of the choice of $x \in \ol{M}$ by identifying
$\Sigma_x$ and $\Sigma_y$ via the coordinate of $\R^l$
($\Phi_2(p)=q$, $\Sigma_x$ and $\Sigma_y$ are actually isometric by Theorem~\ref{th:Bsplit}(iii)).
Consequently, we obtain
\[ \Isom(\wt{M},\wt{F}) =\Isom(\ol{M},\wt{F}) \times \Isom(\Sigma_x,\wt{F}). \]

Since $\Isom(\ol{M},\wt{F})$ is compact,
$\Xi:=\{ \Phi \in \Gamma \,|\, \Phi_2=\Id_{\Sigma_x} \}$
is a finite normal subgroup of $\Gamma=\pi_1(M)$.
The covering
\[ \wt{M}=\ol{M} \times \R^l\ \lra\ \wt{M}/\Xi =:\ol{M}' \times \R^l \]
leads
\[ \Isom(\wt{M}/\Xi,\wt{F}) =\Isom(\ol{M}',\wt{F}) \times \Isom(\Sigma_x,\wt{F}) \supset \Gamma/\Xi, \]
and the projection of $\Gamma/\Xi$ to $\Isom(\Sigma_x,\wt{F})$ is an isomorphism into its image.
As $\Sigma_x$ is flat (Theorem~\ref{th:Bsplit}(ii)),
$\Isom(\Sigma_x,\wt{F})$ is a subgroup of the isometry group $\Isom(\R^l)$
of $\R^l$ with the standard Euclidean metric.
Thus $\Gamma/\Xi$ is isomorphic to a discrete uniform subgroup of $\Isom(\R^l)$.
Therefore $\Gamma/\Xi$ is isomorphic to a crystallographic group of $\R^l$,
and contains $\Z^l$ as a normal subgroup of finite index by the Bieberbach theorem.

(iii)
This is a consequence of (ii).
$\qedd$
\end{proof}

Recall that, if $l=n$, then $\wt{M}$ is a Minkowski normed space.
Our argument was indebted to various fine properties of Berwald spaces,
even at the early stage (see Lemma~\ref{lm:tconv}, Proposition~\ref{pr:Bsplit}).
It is unclear if any of results in this section can be generalized to general Finsler manifolds.

{\small

}

\end{document}